\documentclass[reqno, 12pt, reqno]{amsart}

\usepackage{amsfonts, amsmath, amssymb}
\usepackage{hyperref}
\hypersetup{colorlinks=false}

\usepackage{enumitem}

\usepackage[pdftex]{graphicx}

\usepackage[margin=1.5in]{geometry}

\RequirePackage{mathrsfs} \let\mathcal\mathscr
\numberwithin{equation}{section}

\usepackage{amsmath}%
\usepackage{mathtools}

\usepackage{comment}

\usepackage[capitalise, noabbrev]{cleveref}
\usepackage{colonequals}

\newcommand{\mc}[1]{\mathcal{#1}}
\newcommand{\mf}[1]{\mathfrak{#1}}

\renewcommand{\phi}{\varphi}
\renewcommand{\rho}{\varrho}

\renewcommand{\P}{\mathbb{P}}
\newcommand{\Proj}{\P}

\newcommand{\F}{\mathbb{F}}

\renewcommand{\leq}{\leqslant}
\renewcommand{\geq}{\geqslant}

\renewcommand{\hat}{\widehat}

\theoremstyle{theorem}
\newtheorem{theorem}{Theorem}[section]

\newtheorem{lemma}[theorem]{Lemma}

\theoremstyle{definition}

\newtheorem{remark}[theorem]{Remark}

\numberwithin{equation}{section}

\newcommand{\nid}{\noindent}

\newcommand{\ra}{\rightarrow}

\newcommand{\PP}{\Proj}
\usepackage{tikz}
\usetikzlibrary{3d,calc}
\usetikzlibrary{positioning}

\title{The Bertini Irreducibility theorem for higher codimensional slices}
\author{Philip Kmentt and Alec Shute}
\date{August 2021}
\address{IST Austria\\
Am Campus 1\\
3400 Klosterneuburg\\
Austria}
\email{philipkmentt@hotmail.com}
\email{alec.shute@ist.ac.at}
\begin{document}

\maketitle
\begin{abstract}
In \cite{poonen2020exceptional}, Poonen and Slavov develop a novel approach to Bertini irreducibility theorems over an arbitrary field, based on random hyperplane slicing. In this paper, we extend their work by proving an analogous bound for the dimension of the exceptional locus in the setting of linear subspaces of higher codimensions.
\end{abstract}
\section{Introduction}

Bertini theorems are a family of results which typically state that if a projective variety $X \subseteq \PP^n_{F}$ over a field $F$ has a certain nice property (such as smoothness, or geometric irreducibility), then a generic hyperplane section of $X$ has this property too. We define
\begin{align*} 
\hat{\PP}^n_F &= \{\textrm{hyperplanes }H \subseteq \PP^n_F\}\\
\mathfrak{M}_{\textrm{bad}}^{(1)} &= \{H\in \hat{\PP}^n_F: X\cap H \textrm{ is not geometrically irreducible}\}.
\end{align*}
The set of bad hyperplanes $\mf{M}_{\textrm{bad}}^{(1)}$ is a constructible locus in $\PP^n_F$, and so it makes sense to ask about its dimension. A classical Bertini irreducibility theorem would state that $\dim \mf{M}_{\textrm{bad}}^{(1)} \leq n-1$. However, much more is true, as demonstrated by the following theorem of Olivier Benoist.
\begin{theorem}\cite[Th\'{e}or\`{e}me 1.4]{benoist2011theoreme}\label{benoist}
We have $\dim \mf{M}_{\textrm{bad}}^{(1)} \leq n - \dim X+1$.
\end{theorem}

In 2020, Poonen and Slavov reproved Theorem \ref{benoist} using a novel approach based on estimates for the mean and variance of random hyperplane slices of $X$ over a finite field. In this paper, we demonstrate that Poonen and Slavov's methods can be generalised to deal with linear subspaces $H$ of a general codimension $k$. For a fixed $1\leq  k \leq n-1$, we define
\begin{align}
    V &= \mathbb{G}(n-k,n) = \{\textrm{linear subspaces }H\subseteq \PP^n_F \textrm{ of codimension }k\}\label{definition of V}\\
    \mathfrak{M}_{\textrm{bad}}^{(k)} &= \{H\in V: X\cap H \textrm{ is not geometrically irreducible}\}\label{definition of Mbad}.
\end{align}
The main result of this paper is the following theorem.
\begin{theorem}\label{Main Theorem}
Let $X\subseteq \PP^n_F$ be a geometrically irreducible variety over an arbitrary field $F$. Let $V$ and $\mf{M}_{\textrm{bad}}^{(k)}$ be as in (\ref{definition of V}) and (\ref{definition of Mbad}). Then 
 \begin{equation*}
\dim \mf{M}_{\textrm{bad}}^{(k)} \leq \dim V - \dim X  + k.
 \end{equation*}
\end{theorem}
We recall that for any $k\leq n$, $\dim (\mathbb{G}(n-k,n)) = (n-k+1)k$. Therefore, on taking $k=1$ in Theorem \ref{Main Theorem}, we recover the bound $\dim \mf{M}_{\textrm{bad}}^{(1)} \leq n-\dim X +1$ from Theorem \ref{benoist}.

\begin{remark}\label{bertini can fail}
Suppose that $\dim X \leq k$. If $X$ is not linear, then a generic intersection $X\cap H$ for $H \in V$ will either be empty or a union of at least two points, and so not irreducible. Hence we cannot expect a Bertini irreducibility theorem to hold in this setting. This is consistent with the fact that the bound in Theorem \ref{Main Theorem} becomes trivial when $\dim X \leq k$. 
\end{remark}

In the case $k=1$, Benoist proves in \cite[Proposition 3.1]{benoist2011theoreme} that Theorem \ref{benoist} is best possible, by exhibiting an irreducible variety $X$ such that $\dim \mf{M}_{\textrm{bad}}^{(1)} = n - \dim X+1$. The following theorem, which we prove in Section \ref{proof of sharpness}, is a generalisation of this argument, and shows that Theorem \ref{Main Theorem} is sharp whenever $\dim X \geq k$. 

\begin{theorem}\label{our result is sharp}
Let $1 \leq k \leq r \leq n-1$. Then there exists an irreducible variety $X \subseteq \PP^n$ of dimension $r$ such that 
$$\dim \mf{M}_{\textrm{bad}}^{(k)} = \dim V - r+k.$$
\end{theorem}

\begin{remark}
Rather than considering intersections $X \cap H$, Poonen and Slavov work in \cite{poonen2020exceptional} with a more general setup, where $\phi: X \ra \PP^n$ is a morphism whose nonempty fibres all have the same dimension. They then study the exceptional locus of hyperplanes $H\subseteq \PP^n$ such that $\phi^{-1}(H)$ is not geometrically irreducible. We expect that Theorem \ref{Main Theorem} could easily be extended to their setting.
\end{remark}

\nid \textbf{Acknowledgements.} The authors would like to thank Tim Browning for suggesting this project and for many helpful discussions during the development of this paper.

\section{Statistics of random hyperplane slices}\label{section mean and variance}
In this section, we work over a fixed finite field $\F_q$, and estimate the mean and variance of the number of $\F_q$-points on random hyperplane slices of $X$. 

\begin{lemma}\label{mean variance} 
Fix a variety $X \subseteq \mathbb{P}^{n}_{\mathbb{F}_q}$. Let $V= \mathbb{G}(n-k,n)$. For $H \in V(\mathbb{F}_q)$ chosen uniformly at
random, define the random variable $Z := \#(X\cap H)(\mathbb{F}_q)$. Let $\mu$ and $\sigma^2$ denote the mean and variance of $Z$ respectively. Then
\begin{align*}
 \mu &= \#X(\mathbb{F}_q)(q^{-k}+O(q^{-k-1}))\\
\sigma^{2} &= O(q^{-k}\#X(\mathbb{F}_q)).
\end{align*}
\end{lemma}
\begin{proof}
We begin by considering the mean. We have 
\begin{align}
\mu &= \frac{1}{\#V(\mathbb{F}_q)}\sum_{H \in V(\mathbb{F}_q)}\#(X\cap H)(\mathbb{F}_q)\nonumber\\
&= \frac{1}{\#V(\mathbb{F}_q)}\sum_{H \in V(\mathbb{F}_q)} \sum_{x \in (X \cap H)(\F_{q})} 1\nonumber\\
&=
\frac{1}{\#V(\mathbb{F}_q)} \sum_{x \in X(\mathbb{F}_q)} \sum_{\substack{H \in V(\mathbb{F}_q) \\ x \in H(\mathbb{F}_q)}}1.\label{expression for mean}
\end{align}
We observe that the inner sum in (\ref{expression for mean}) is independent of $x$. Consequently, we can average the value of this sum over a dummy variable $w\in \PP^n(\F_q)$ to obtain
\begin{align*}
\sum_{\substack{H \in V(\mathbb{F}_q) \\ x \in H(\mathbb{F}_q)}}1  = \frac{1}{\#\mathbb{P}^{n}(\mathbb{F}_q)}
\sum_{w \in \mathbb{P}^{n}(\mathbb{F}_q)}\sum_{\substack{H \in V(\mathbb{F}_q) \\ w \in H(\mathbb{F}_q)}}1.
\end{align*}
Returning to (\ref{expression for mean}), we conclude that
\begin{align}
\mu &= \frac{1}{\#\mathbb{P}^{n}(\mathbb{F}_q)\#V(\mathbb{F}_q)}\sum_{x \in X(\mathbb{F}_q)} \sum_{H \in V(\mathbb{F}_q)} \sum_{w \in H(\mathbb{F}_q)}1 \nonumber\\
&= \frac{\#X(\mathbb{F}_q)\#\mathbb{P}^{n-k}(\mathbb{F}_q)}{\#\mathbb{P}^{n}(\mathbb{F}_q)}\label{computation of mean}\\
&=\#X(\mathbb{F}_q)(q^{-k}+O(q^{-k-1}))\label{final estimate for mean}.
\end{align}
A similar argument can be applied for the variance. We have
\begin{align}
\sigma^{2} &= \left(\frac{1}{\#V(\mathbb{F}_q)}\sum_{H \in V(\mathbb{F} _q)}(\#(X\cap H)(\mathbb{F}_q))^{2}\right) - \mu^{2}\nonumber\\
&= \left(\frac{1}{\#V(\mathbb{F}_q)}\sum_{H \in V(\mathbb{F}_q)} \sum_{\substack{x, y \in (X\cap H)(\mathbb{F}_q)}} 1\right) - \mu^{2}\nonumber\\
&=
\left(\frac{1}{\#V(\mathbb{F}_q)} \sum_{x, y \in X(\mathbb{F}_q)} \sum_{\substack{H \in V(\mathbb{F}_q) \\ x,y \in H(\mathbb{F}_q)}}1\right) - \mu^{2}.\label{expression for variance}
\end{align}
The contribution to (\ref{expression for variance}) from the case $x=y$ is simply $\mu$. Therefore, $\sigma^2 = B-\mu^2 + \mu$, where
\begin{equation}\label{definition of B}
B = \frac{1}{\#V(\mathbb{F}_q)} \sum_{\substack{x, y \in X(\mathbb{F}_q) \\x \ne y }} \sum_{\substack{H \in V(\mathbb{F}_q) \\ x,y \in H(\mathbb{F}_q)}}1.
\end{equation}
A similar trick to above can be applied to the inner sum of (\ref{definition of B}), by averaging over two dummy variables $u$ and $v$, ranging over the entirety of $ \mathbb{P}^{n}(\mathbb{F}_q)$, but this time, with the added condition $u \neq v$. We obtain 
\begin{align}
B&= \frac{1}{\#V(\mathbb{F}_q)\#\mathbb{P}^{n}(\mathbb{F}_q)(\#\mathbb{P}^{n}(\mathbb{F}_q)- 1 )}  \sum_{\substack{x, y \in X(\mathbb{F}_q) \\x \ne y }} \sum_{\substack {u,v \in \mathbb{P}^{n}(\mathbb{F}_q)\\ u \ne v }} \sum_{\substack{H \in V(\mathbb{F}_q) \\ u, v \in H(\mathbb{F}_q) }}1
\nonumber\\
&= \frac{1}{\#V(\mathbb{F}_q)\#\mathbb{P}^{n}(\mathbb{F}_q)(\#\mathbb{P}^{n}(\mathbb{F}_q)- 1 )}  \sum_{\substack{x, y \in X(\mathbb{F}_q) \\x \ne y }} \sum_{{H \in V(\mathbb{F}_q)}} \sum_{\substack {u, v \in H(\mathbb{F}_q) \\ u \ne v }}1
\nonumber\\
&= \frac{\# X(\mathbb{F}_q) (\# X(\mathbb{F}_q)-1)\#\mathbb{P}^{n-k}(\mathbb{F}_q)(\#\mathbb{P}^{n-k}(\mathbb{F}_q)- 1)}
{\#\mathbb{P}^{n}(\mathbb{F}_q)(\#\mathbb{P}^{n}(\mathbb{F}_q)- 1 )}\label{computation of B}.
\end{align}
Now, inspecting every term in (\ref{computation of B}) and comparing it to the corresponding term in the expression for $\mu^2$ obtained from squaring (\ref{computation of mean}), we see that $B \leq \mu^{2},$ implying that $\sigma^{2} \leq \mu$. Combining with (\ref{final estimate for mean}), we deduce that in particular, $\sigma^2 =O(q^{-k}\#X(\mathbb{F}_q))$.
\end{proof}

\section{Proof of Theorem \ref{Main Theorem}}
The following classical theorem of Lang and Weil provides an estimate for the number of $\F_q$-points on a projective variety $X$.

\begin{lemma}[\cite{lang1954number}]\label{lang weil}
Let $X$ be a projective variety over $\mathbb{F}_q$ of dimension $r$. Let $a$ be the number of irreducible components of $X$ that are geometrically irreducible and have dimension $r$. Then $\#X(\mathbb{F}_q) = aq^{r} + O(q^{r-1/2} )$.
\end{lemma}

We can now deduce Theorem \ref{Main Theorem} from Lemma \ref{mean variance} and Lemma \ref{lang weil} by following a similar argument to \cite[Sections 5,6]{poonen2020exceptional}. By applying Poonen and Slavov's reduction from \cite[Section 3]{poonen2020exceptional}, it suffices to work over a finite ground field $\F_q$. Moreover, by passing to a finite extension if necessary, we may assume that $q$ is sufficiently large. 


We call $H \in V(\F_q)$ \textit{very bad} if the number of $\F_q$-irreducible components of $X\cap H$ which are geometrically irreducible is not 1. Let $\mc{A}\subseteq V(\F_q)$ denote the set of very bad linear spaces $H$. The strategy will be to deduce Theorem \ref{Main Theorem} from appropriate upper and lower bounds for $\#\mc{A}$. Applying a similar argument to \cite[Lemma 6.2]{poonen2020exceptional}, we have for sufficiently large $q$ that
\begin{equation}\label{lower bound for A}
\#\mc{A}\gg \#\mathfrak{M}_{\textrm{bad}}^{(k)}(\F_q).
\end{equation}

In order to obtain an upper bound for $\#\mc{A}$, the idea is to show that if $H \in \mc{A}$, then the random variable $Z=\#(X\cap H)(\F_q)$ introduced in Section \ref{section mean and variance} differs considerably from the mean $\mu$. Hence there cannot be many such $H$ by the upper bound for the variance obtained in Lemma \ref{mean variance}. 

Let $r= \dim X$. Combining Lemma \ref{mean variance} and Lemma \ref{lang weil}, we have 
\begin{align*}
\mu &= q^{r-k} + O(q^{r-k-1/2})\\
\sigma^{2} &= O(q^{r-k}).
\end{align*}
If $H \in \mc{A}$, then by Lemma \ref{lang weil}, $\#(X \cap H)(\F_q)$ is either $O(q^{r-k-1/2})$ or at least $2q^{r-k}-O(q^{r-k-1/2})$. Consequently, 
\begin{align*}
\mid \# (X \cap H)(\mathbb{F}_q) - \mu \mid \geq q^{r-k}-O(q^{r-k-\frac{1}{2}}) \geq \frac{1}{2}q^{r-k},
\end{align*}
for sufficiently large $q$. Define $t$ such that $\frac{1}{2}q^{r-k} = t\sigma$. Then,
\begin{align*}
\textrm{Prob}(H\in \mc{A}) &\leq \textrm{Prob}(\mid \# (X \cap H)(\mathbb{F}_q) - \mu \mid \geq t\sigma)\\
&\leq \frac{1}{t^{2}} \quad\textrm{ (by Chebyshev's inequality)}\\
&= \frac{4\sigma^{2}}{q^{2(r-k)}}\\
&= O(q^{-r+k}).
\end{align*}
Multiplying by $\#V(\F_q)$, we obtain
\begin{equation}\label{upper bound for A}
\#\mc{A}= O(q^{\dim V - r + k}).
\end{equation}
Therefore, combining (\ref{lower bound for A}) and (\ref{upper bound for A}), we conclude that 
\begin{align*}
\dim \mathfrak{M}_{\textrm{bad}}^{(k)} \leq \dim V - r + k.
\end{align*}

\section{Proof of Theorem \ref{our result is sharp}}\label{proof of sharpness}

If $r=1$, then $k=1$, and so by Remark \ref{bertini can fail} it suffices to take $X$ to be any curve of degree at least $2$. From now on, we assume that $r\geq 2$. We use the same construction of $X$ as in \cite[Proposition 3.1]{benoist2011theoreme}, which we now recall for convenience. Fix an integral curve $C$ of degree at least $2$, and a linear space $L$ of dimension $r-2$ not containing $C$. Let $X\subseteq \PP^n$ be a cone with base $C$ and vertex $L$. Then $X$ is an irreducible variety of dimension $r$. 

Let $\mc{H}$ denote the locus of hyperplanes $H$ which contain $L$. We denote by $\mc{U}$ the locus of hyperplanes $H \in \mc{H}$ satisfying the following additional properties. 
\begin{enumerate}
    \item $H\cap C$ contains a point $P$ disjoint from $L$.
    \item $\dim(X \cap H) = r-1$.
    \item $\deg(X \cap H) \geq 2$.
\end{enumerate}
These properties hold for a generic $H\in \mc{H}$, so $\dim \mc{U} = \dim\mc{H}$. Let $N$ denote the linear span of $L$ and $P$. Suppose that $H\in \mc{U}$. Then $N \subseteq H$. Furthermore, $X$ contains all lines between $P$ and $L$, and hence also contains $N$. Therefore $N\subseteq X\cap H$. From the above properties, we deduce that $N$ is a proper closed subset of $X\cap H$ with $\dim N = r-1 = \dim (X \cap H)$, and hence $X\cap H$ is not irreducible. We conclude that 
$$\dim \mf{M}_{\textrm{bad}}^{(1)} \geq \dim \mc{U} =\dim \mc{H} = n-r+1,$$
and in fact, equality holds by Theorem \ref{benoist}.

We now generalise to an arbitrary $k\leq r$. Fix $H \in \mc{U}$. Let $\mc{M}_H$ denote the locus of linear spaces $M\in \mathbb{G}(n-k,n)$ satisfying $M \subseteq H$. Then 
\begin{equation}\label{dimension of MH}
\dim \mc{M}_H = \dim(\mathbb{G}(n-k,n-1))= \dim(\mathbb{G}(n-k,n)) - (n-k+1).
\end{equation}

For a generic $M \in \mc{M}_H$, we have 
\begin{enumerate}
    \item $\dim(X \cap M) = r-k = \dim(N \cap M)$.
    \item $\deg(X \cap M) \geq 2$.
    \item The linear span of $M$ and $L$ is $H$. 
\end{enumerate}
For such $M$, we have that $N\cap M$ is a proper closed subset of $X\cap M$ with the same dimension as $X\cap M$, and hence $M\in \mf{M}_{\textrm{bad}}^{(k)}$. Property (3) ensures that $H$ is the unique element of $\mc{U}$ containing $M$. Combining with (\ref{dimension of MH}), we deduce that 
\begin{align}
    \dim \mf{M}_{\textrm{bad}}^{(k)} 
    &\geq \dim(\mathbb{G}(n-k,n)) - (n-k+1) + \dim \mc{U} \label{actually an equality}\\
    &= \dim(\mathbb{G}(n-k,n)) + k - r,\nonumber
\end{align}
and from Theorem \ref{Main Theorem}, (\ref{actually an equality}) is in fact an equality.

\bibliographystyle{plain}
\bibliography{references.bib}

\end{document}